\documentclass{amsproc}

\usepackage{amssymb}
\usepackage{tikz}
\usepackage{pgfplots}
\usepackage{multirow}

\usepackage{graphicx}
\usepackage{amsmath}
\usepackage{subfigure}
\usepackage{calc}
\usepackage{float}
\usepackage{wrapfig}

\newtheorem{lemma}{Lemma}[section]

\newtheorem{question}{Question}[section]

\usetikzlibrary{3d}

\def\Z{\mathbb{Z}}
\def\R{\mathbb{R}}

\def\H{\mathbb{H}}
\def\K{\mathbb{K}}
\def\M{\mathbb{M}}
\def\X{\mathbb{X}}
\def\Y{\mathbb{Y}}
\def\J{\mathbb{J}}
\def\N{\mathbb{N}}

\def\F{\mathcal{F}}

\def\0{{\bf{0}}}

\def\la{\langle}
\def\ra{\rangle}

\def\-{{\mbox{\tiny $ - $ }}}
\def\+{{\mbox{\tiny $ + $ }}}

\def\0{{\mbox{\tiny $ (0) $ }}}
\def\1{{\mbox{\tiny $ (1) $ }}}
\def\2{{\mbox{\tiny $ (2) $ }}}

\numberwithin{equation}{section}

\usepackage{pstricks}
\definecolor{gray7}{rgb}{0.7,0.7,0.7}
\definecolor{gray8}{rgb}{0.8,0.8,0.8}
\definecolor{gray9}{rgb}{0.9,0.9,0.9}

\usepackage{lipsum}

\newcommand\blfootnote[1]{
  \begingroup
  \renewcommand\thefootnote{}\footnote{#1}
  \addtocounter{footnote}{-1}%
  \endgroup
}

\begin{document}


\allowdisplaybreaks

\begin{center}
{\Large  \textsc{The Answer to Baggett's Problem is Affirmative}}
\end{center}
\bigskip

\begin{center}
      \large\textit{Xingde Dai}
\textsf{}\end{center}

\email{xdai@uncc.edu}
\address{The University of North Carolina at Charlotte}


\allowdisplaybreaks

{\begin{center}
\textsf{\textsf{Dedicated to Yuan Qing Lan}}
\end{center}}

\vskip 10pt
{\small
\begin{center}
\textbf{Abstract}\\
\end{center}
Let $\psi$ be a Parceval wavelet in $L^2 (\R)$ with the space of negative dilates $V(\psi)$. The intersection of the dilates $V(\psi)$
is the zero space. In other words, we have
\begin{align*}
    \bigcap_{n\in\Z} D^n \overline{\textrm{span}}\{D^{\textrm{-}m} T^\ell \psi \mid m\geq 0, m,\ell\in\Z\}=\{0\}.
\end{align*}

\vskip 20pt

}

\section{Introduction}

Denote $L^2 (\R)$ as $\H$. Let $B(\H)$ denote the space of bounded linear operators acting on $\H$. Let $T$, $D$ and $\F$ be the translation, dilation and Fourier transform operators defined as follows. For  $f\in\H$,
\begin{eqnarray*}
  (Tf) (t) &=& f(t-1).\\
  (Df) (t) &=& \sqrt{2} f(2t).\\
  (\F f) (s) &=& \widehat{f} (s) = \frac{1}{\sqrt{2\pi}} \int_\R e^{-i t s } f(t) dt.
\end{eqnarray*}
The operators $T$, $D$ and $\F$ are unitary operators.
We have
 $ (\F^{-1} f) (t) = \check{f} (s) = \frac{1}{\sqrt{2\pi}} \int_\R e^{i  st } f(t) dt.$
For $A$ in $B(\H)$, let $\widehat{A}\equiv\F A\F^{-1}.$ We have $\F A =\F A \F^{-1} \F=\widehat{A}\F$.
The operator $\widehat{T}$ is the multiplication operator  $M_{e^{-is}}$ which maps  $f(s)$ to $e^{-is}f(s), f\in\H$.
Also $\widehat{D}^n = D^{-n}, n\in\Z$. see \cite{dailarson}

A set $\{\vec{x}_n \mid n\in \J \}$ in $\H$  is a \emph{normalized tight frame }for $\H$ if for each
$\vec{x}\in\H$
\begin{align*}
    \|\vec{x}\|^2 = \sum_{n\in\J} |\langle \vec{x}, \vec{x}_n \rangle|^2.
\end{align*}
Here the index set $\J$ is countable infinite.
An orthonormal basis for $\H$ is a normalized tight frame for $\H$,  and not vice versa.

A function $\psi\in L^2(\R)$ is called a \textit{Parseval wavelet} if the set
$\{D^mT^\ell \psi \mid (m,\ell)\in\Z^2\}$ forms a normalized tight frame for $L^2(\R)$.
In addition, if the set $\{D^nT^\ell \psi \mid (n,\ell)\in\Z^2\}$ is orthogonal,
then $\psi$ must be a unit vector and it is an orthonormal wavelet.
For $\vec{x}\in L^2(\R)$, we will use notation $V(\vec{x})$ as
\begin{align*}
    V(\vec{x}) \equiv \overline{\textrm{span}}\{D^{\textrm{-}m} T^\ell x \mid m\in\Z, m\geq 0, \ell\in\Z\}.
\end{align*}
Some authors called $V(\vec{x})$ the space of negative dilates. \cite{bownik1}

In 1999 Larry Baggett asked the following question.
\begin{question}\label{question:larry}(Baggett, 1999)
Let $\psi\in L^2(\R)$ be a Parseval wavelet. Is the following Equation \eqref{eq:baggett} holds?
\begin{align}\label{eq:baggett}
    \bigcap_{n\in\Z} D^n\left(V(\psi)\right)=\{0\}.
\end{align}
\end{question}
An affirmative answer to the Question \ref{question:larry} implies that every Parseval wavelet is associated with the
general multiresolution analysis (GMRA). The concept of GMRA is introduced by Baggett, Medina and Merril \cite{baggett} as
a natural generalization to the concept multiresolution analysis (MRA).
The question \ref{question:larry} is posted as an open question in \cite{bownik}.

In this paper we prove that the answer to Baggett's question is ``Yes".
Our reasoning is for $L^2(\R)$ case. However, the idea for the proof works for the general case $L^2(\R^d)$.

\section{Proof}

A sequence $\{\vec{x}_n\mid n\in\J\}$ in a (separable) Hilbert space (Banach space) $\X$ is called a Schauder
basis of $\X$ if for every $\vec{x} \in \X$ there is a unique sequence of scalars $\{a_n\mid n\in\J\}$ so that
\begin{align*}
\vec{x}=\sum_{n\in\J} a_n \vec{x}_n.
\end{align*}
The convergence is in the norm of $\X$. We will call the set of numbers $\{a_n\mid n\in\J\}$ the basis coefficients associated with $\vec{x}$. In addition, if for an arbitrary  permutation $\pi$ of $\J$ we have
\begin{align*}
\vec{x}=\sum_{n\in\J} a_{\pi(n)} \vec{x}_{\pi(n)},
\end{align*}
we will call the above Schauder basis  an \textit{unconditional basis}. An orthonormal basis of a Hilbert space is an unconditional basis.
Let $\iota$ be a isomorphism (hence continuous, by the Open mapping theorem) from a Hilbert space $\X$ onto a Hilbert space $\Y$ and let $\{\vec{x}_n \mid n\in\J\}$ be an orthonormal basis of $\X$. Since the isomorphism $\iota$ maps a cauchy sequence in $\X$ to a cauchy sequence in $\Y$, the sequence $\{\vec{y}_n\equiv \iota(\vec{x}_n) \mid n\in\J\}$ is an unconditional basis in $\Y$.
Let $\{\vec{x}_n\mid n\in\J\}$ be a Schauder basis of a Hilbert space $\X$. Then there exist corresponding
linear functionals
$\vec{x}^* _n, n\in\J$ in $\X^*$ so that
\begin{align*}
    x^* _n (\vec{x}_m) =  \delta_{n,m}, n,m\in\J.
\end{align*}
The notation $\delta$ is the Kronecker delta.
Let $\H=L^2(\R)$ with an orthonormal wavelet $\eta$. We will view the orthonormal basis $\{D^j T^\ell\eta \mid (j,\ell)\in\Z^2\}$, since it is an
unconditional basis as $\{\vec{e}_n \mid n\in\J\}$ for $\J=\Z^2$ in one stream. For the basis, we refer \cite{heil} to the reader.

We will need the followsing Lemma \ref{lm:hanlarson} by Han and Larson.
\begin{lemma}\label{lm:hanlarson}(Han,Larson \cite{hanlarson})
Let  $\{\vec{x}_n \mid n\in \J \}$ be a normalized tight frame for $\H$.
Then there exists a Hilbert space $\M$ with a normalized tight frame $\{\vec{m}_n \mid n\in\J\}$ for $\M$ such that
the set $\{ \vec{m}_n \oplus \vec{x}_n\mid n\in\J\}$ forms an orthonormal basis for $\M\oplus\H$.
\end{lemma}
In above Lemma \ref{lm:hanlarson} if the set $\{\vec{x}_n \mid n\in \J \}$
is an orthonormal basis for $\H$, then $\M$ is a $0$ space and $\{\vec{m}_n \mid n\in \J \}$ is the set of zero vectors.

Let $\psi$ be a given Parseval wavelet for $\H=L^2(\R)$.
By Lemma \ref{lm:hanlarson} there exists a separable Hilbert space $\M$
with a normalozed tight frame $\{\vec{m}_{n,\ell} \mid (n,\ell) \in\Z^2\}$ such that
the set
\begin{align*}
    \{\vec{m}_{n,\ell}\oplus D^nT^\ell \psi \mid (n,\ell)\in\Z^2\}
\end{align*}
forms an orthonormal basis for $ \M\oplus \H$. Denote $\vec{e}_{n,\ell } \equiv \vec{m}_{n,\ell}\oplus D^nT^\ell \psi$.
Let $\vec{x}\in \H$.  We have
\begin{align*}
0\oplus \vec{x}   &=\sum_{n,\ell\in \Z} \langle 0\oplus \vec{x}, \vec{m}_{n,\ell}\oplus D^nT^\ell \psi \rangle \vec{e}_{n,\ell }\\
            &=\sum_{n,\ell\in \Z} \langle \vec{x}, D^nT^\ell \psi \rangle \ (0\oplus D^nT^\ell \psi)\\
            &=0\oplus \left(\sum_{n,\ell\in \Z} \langle \vec{x}, D^nT^\ell \psi \rangle D^nT^\ell \psi\right).
\end{align*}
So, we have the well known equation,
\begin{align}\label{eq:dd}
    \vec{x}=\sum_{n,\ell\in \Z} \langle \vec{x}, D^nT^\ell \psi \rangle D^nT^\ell \psi, \forall \vec{x} \in \H.
\end{align}
This is equivalent to the definition of the Parceval wavelet.\\

Let $\eta$ be the function defined as
\begin{align*}
    \widehat{\eta}=\frac{1}{\sqrt{2\pi}}\chi_{_{[-2\pi,-\pi]\cup [\pi,2\pi]}},
\end{align*}
where $\chi$ is the characteristic function.
It is well known that the set $\{D^nT^\ell \eta \mid (n,\ell)\in\Z^2\}$ is an orthonormal basis of $L^2 (\R)$.
The function $\eta$ is the Littlewood-Paley wavelet.
An element $\vec{x}\in L^2 (\R)$ is in the form $\vec{x}=\sum_{n,\ell\in \Z} \langle \vec{x}, D^nT^\ell \eta \rangle D^nT^\ell \eta$.
Define a mapping $U:\H \rightarrow \M \oplus \H$ as
\begin{align*}
    U\vec{x} &= U\left(\sum_{n,\ell\in\Z} \langle \vec{x}, D^nT^\ell \eta \rangle D^nT^\ell \eta\right)
    \equiv \sum_{n,\ell\in\Z} \langle \vec{x}, D^nT^\ell \eta \rangle \vec{e}_{n,\ell}\\
    &=\sum_{n,\ell\in\Z} \langle \vec{x}, D^nT^\ell \eta \rangle (\vec{m}_{n,\ell} \oplus D^nT^\ell \psi).\\
\end{align*}
The operator $U$ maps the orthonormal basis $\{D^nT^\ell \eta \mid (n,\ell)\in\Z^2\}$ of $\H$ to the
orthonormal basis $\{\vec{e}_{n,\ell}\mid (n,\ell)\in\Z^2 \}=\{ \vec{m}_{n,\ell}\oplus D^nT^\ell \psi \mid (n,\ell)\in\Z^2\}$ of $\M\oplus \H$.
This is a unitary operator.
Let $P$ be the orthogonal projection from $\M \oplus \H$ to the subspace $0\oplus \H$.
Let $I_0$ denote the mapping sending $0\oplus f$ in $0\oplus \H$ to $f$ in $\H$.
Define the operator $\Xi$ as
\begin{align}
\Xi  =I_0PU .
\end{align}
The operator $\Xi$ is a bounded linear operator, i.e $\Xi\in B(\H)$.
We have
\begin{align*}
   \Xi(\vec{x})&=\Xi\left(\sum_{n,\ell\in\Z} \langle \vec{x} , D^nT^\ell \eta \rangle D^nT^\ell \eta\right)\\
   &= I_0P\left(U\sum_{n,\ell\in\Z} \langle \vec{x} , D^nT^\ell \eta \rangle D^nT^\ell \eta \right)\\
   &= I_0 \left(P\sum_{n,\ell\in\Z} \langle \vec{x} , D^nT^\ell \eta \rangle \ (\vec{m}_{n,\ell}\oplus D^nT^\ell \psi)\right)\\
    &= I_0\sum_{n,\ell\in\Z} \langle \vec{x} , D^nT^\ell \eta \rangle \ (0\oplus D^nT^\ell \psi)\\
    &=\sum_{n,\ell\in\Z} \langle \vec{x} , D^nT^\ell \eta \rangle D^nT^\ell \psi.
\end{align*}
We obtain
\begin{align}\label{eq:def}
    \Xi\left(\sum_{n,\ell\in\Z} \langle \vec{x} , D^nT^\ell \eta \rangle D^nT^\ell \eta\right)
    &=\sum_{n,\ell\in\Z} \langle \vec{x} , D^nT^\ell \eta \rangle D^nT^\ell \psi.
\end{align}
By Equation \eqref{eq:def} when  $\vec{x} =\eta$,
we have $\Xi (\eta)=\psi$. When $\vec{x} =D^n T^\ell \eta$, we have
\begin{align}\label{eq:pcommutant}
\Xi (D^nT^\ell\eta)=D^nT^\ell \psi=D^nT^\ell \Xi(\eta), \forall (n,\ell)\in\Z^2.
\end{align}
The operator $\Xi$ \textit{commute} with the unitary system $\{D^n T^\ell\mid (n,\ell)\in\Z^2\} $ \textit{at}
the point $\eta$. The collection of all operator with this properties is called the \emph{point commutant at} $\eta$, which is denoted as $C_\eta(D,T)$ in \cite{dailarson}.
$ \Xi\in C_\eta(D,T)$. Also
$(\Xi D)D^nT^\ell\eta=\Xi D^{n+1}T^\ell\eta=D^{n+1}T^\ell \Xi \eta=(D\Xi)D^nT^\ell\eta$. Since $\{D^nT^\ell \eta\}$ is an orthonormal basis, so
 we have
\begin{align}\label{eq:comm}
D\Xi=\Xi D \textrm{ and } D^{-1}\Xi=\Xi D^{-1} .
\end{align}

Let $\vec{y} $ be an arbitrarily given element in $\H$. Let $\vec{x} =\Xi^* \vec{y} $.
Then
\begin{align*}
    \Xi (\vec{x})  &= \sum_{n,\ell\in\Z} \langle \vec{x} , D^nT^\ell \eta \rangle D^nT^\ell \psi\\
    &=\sum_{n,\ell\in\Z} \langle \Xi ^*\vec{y} , D^nT^\ell \eta \rangle D^nT^\ell \psi\\
    &=\sum_{n,\ell\in\Z} \langle \vec{y} , \Xi D^nT^\ell \eta \rangle D^nT^\ell \psi\\
    &=\sum_{n,\ell\in\Z} \langle \vec{y} , D^nT^\ell \psi \rangle D^nT^\ell \psi\\
    &=\vec{y} ,
\end{align*}
by Equation \eqref{eq:dd}. So $\Xi$ is surjective.
\begin{align}\label{eq:xih}
    \Xi(\H) = \H.
\end{align}
By the Open mapping theorem the operator $\Xi$ is
an open mapping.

Denote the kernel of $\Xi$ as $N$ and denote the orthogonal projection to $N$ as $Q$. Denote the orthogonal complement of $N$
as $\K\equiv N^\perp$ and denote the projection to $\K$ as $Q^\perp$. It is clear that $Q^2=Q$ and $(Q^\perp) ^2=Q^\perp $. We have $\K=N^\perp=Q^\perp \H=Q^\perp L^2(\R)$. By the Open mapping theorem, the operator $\Xi_{|\K}$
is a continuous isomorphism from $\K$ onto $\H$. We denote $\Xi_{|\K}$ as $\iota$ and denote its inverse as $\kappa$.
\begin{align*}
    \K      &=  Q^\bot \H.\\
    \H      &=  N\oplus \K = N \oplus Q^\bot \H.\\
    \iota   &=  \Xi_{|\K}      : \K\rightarrow \H.\\
    \kappa  &=  \iota^{-1}  : \H\rightarrow \K.
\end{align*}

For $n\in\Z$,
we denote the closed linear span of the orthonormal set $\{D^n T^\ell\eta \mid \ell\in\Z\}$ as $W_n$. Notice that $D$ is a unitary operator, we have
\begin{align*}
    W_n=\overline{\textrm{span}}\{D^n T^\ell\eta \mid \ell\in\Z\}=D^n W_0.
\end{align*}
Since $\eta$ is an orthogonal wavelet, the subspaces $\{W_n \mid n\in\Z\}$ are mutually orthogonal to each other.
Let
\begin{align*}
    V_n = \overline{\textrm{span}}\{D^j T^\ell\eta \mid j\in \Z, j\leq n; \ell\in\Z \}.
\end{align*}
We have
\begin{align*}
V_0 &= V(\eta) =\bigoplus_{j\leq 0, j\in\Z} W_j,\\
    V_n&=D^{n}V_0 =\bigoplus_{j\leq n, j\in\Z} W_j,
\end{align*}
and
\begin{align*}
    \cdots  &\subset  V_{-1 }\subset  V_{0} \subset  V_{1} \subset \cdots.
\end{align*}

$N,\K$ are subspaces of $\H$. In Lemma \ref{lm:relations} we will discuss relations of $D$ with $N, N^\perp, Q , \iota$ and $\kappa$.
\begin{lemma}\label{lm:relations}
\begin{enumerate}
            \item $DN=N$.
            \item $DN^\perp=N^\perp$ or $D\K=\K$.
            \item $ DQ=QD$.
            \item $ DQ^\perp=Q^\perp D$.
            \item $D\iota=\iota D$.
            \item $D\kappa=\kappa D$.
\end{enumerate}
\end{lemma}
\begin{proof}
(1). Let $\vec{x}\in N$. Then $\Xi (D\vec{x})=D(\Xi \vec{x})=0$, and
$\Xi (D^{-1}\vec{x})=D^{-1}(\Xi \vec{x})=0$.
 So $D\vec{x}\in N$ and
$D^{-1}\vec{x}\in N$. We have $DN=N$.

(2) Let $\vec{y}\in \K \subset \H$. Then $\vec{y}\in\K $ iff $\vec{y} \perp N.$ Consider $D\vec{y}$ and $\vec{x}\in N$,
\begin{align*}
    \la D\vec{y},  \vec{x} \ra &=\la \vec{y},  D^* \vec{x} \ra=\la \vec{y},  D^{-1} \vec{x} \ra=0,
\end{align*}
since $\vec{y}\in\K$ by assumption and $D^{-1} \vec{x}\in N$ by (1). So $D\vec{y} \perp N$, or $D\vec{y} \in \K$,
$D\K\subset \K$. When we replace $D\vec{y}$ by $D^{-1} \vec{y} (= D^* \vec{y})$, the above reasoning will show that
$D^{-1}\K\subset \K$, which equivalent to $\K\subset D\K$. So we have $D\K= \K$.

(3) Let $\vec{x}\in\H$. Then $\vec{x}=f+ g$ for $f=Q\vec{x}$ and $g=Q^\perp \vec{x}$. So $Qg=0, Qf=f$.
So
\begin{align*}
    DQ \vec{x} =Df.
\end{align*}
Notice that $   Df\in DN=N$ and $D g\in D  N^\perp=N^\perp=\K$. We have $QDg=0$ and $QDf=Df$.
\begin{align*}
    QD \vec{x} =Q\left(Df+ Dg\right)=QDf=Df=DQ \vec{x}.
\end{align*}
So, $DQ=QD$.

(4) Similar as (3).

(5) By (4)
\begin{align*}
D\iota= D\Xi Q^\perp=\Xi D Q^\perp=\Xi Q^\perp D=\iota D.
\end{align*}

(6) By (5) we have $D\iota=\iota D$. This is true iff $(D\iota)^{-1}=(\iota D)^{-1}$
iff $\iota^{-1}D^{-1}=D^{-1}\iota^{-1}$
iff $D\kappa=\kappa D$.

\end{proof}

Consider the subspace $W_0$ of $\H$. It is infinite dimensional. We have
\begin{align*}
    W_0=QW_0 \oplus Q^\perp W_0.
\end{align*}
\begin{lemma}\label{lm:Wzero}
\begin{align}
Q^\bot W_0\neq \{0\}.
\end{align}
\end{lemma}
\begin{proof}
We prove by contradiction. Assume $Q^\bot W_0 =\{0\}.$ This implies that
\begin{align*}
    W_0&\subset QW_0\subset N.
\end{align*}
So, for $n\in\Z$,
\begin{align*}
    W_n= D^n W_0 \subset D^n QW_0 \subset Q D^n W_0=QW_n\subset N.
\end{align*}
This implies that $\H \subset N$, or
\begin{align*}
\Xi (\H)=\{0\}.
\end{align*}
A contradiction to Equation \eqref{eq:xih}.

\end{proof}

Let
\begin{align*}
\{\vec{a}_{0,i} \mid i\in \J_1\}\textrm{ and }
\{\vec{b}_{0,i} \mid i\in \J_2\}
\end{align*}
be orthonormal base for  $QW_0$ and $Q^\perp W_0$, respectively.
Here $\J_1$ and $\J_2$ are subset of $\N$, the counting numbers.
By Lemma \ref{lm:Wzero} $\J_2\neq \emptyset$.
Consider the disjoint union $\{\vec{a}_{0,i} \mid i\in \J_1\}\cup\{\vec{b}_{0,i} \mid i\in \J_2\}$.
This is a countably infinite  set since $W_0$ is infinite dimensional.
Also, this is an orthonormal set.
We reorder it and write it as
\begin{align*}
    \{\vec{\lambda}_{0,i} \mid i\in\N\}=\{\vec{a}_{0,i} \mid i\in \J_1\}\cup\{\vec{b}_{0,i} \mid i\in \J_2\}.
\end{align*}
For a point $\vec{x} \in W_0$,
\begin{align*}
\vec{x} &= Q\vec{x}+Q^\perp \vec{x}\\
        &= \sum_{i\in\J_1} \la Q\vec{x}, \vec{a}_{0,i}  \ra \vec{a}_{0,i} + \sum_{i\in\J_2} \la Q^\perp\vec{x}, \vec{b}_{0,i}  \ra \vec{b}_{0,i}\\
        &= \sum_{i\in\J_1} \la \vec{x}, \vec{a}_{0,i}  \ra \vec{a}_{0,i} + \sum_{i\in\J_2} \la \vec{x}, \vec{b}_{0,i}  \ra \vec{b}_{0,i}.
\end{align*}
Thus
\begin{align*}
       \vec{x} = \sum_{i\in\J_1} \la \vec{x}, \vec{a}_{0,i}  \ra \vec{a}_{0,i} + \sum_{i\in\J_2} \la \vec{x}, \vec{b}_{0,i}  \ra \vec{b}_{0,i}
       = \sum_{i\in\N} \la \vec{x}, \vec{\lambda}_{0,i}  \ra \vec{\lambda}_{0,i}.
\end{align*}
So the set $\{\vec{\lambda}_{0,i} \mid i\in\N\}$ is an orthonormal basis for $W_0$.
This follows that the set $D^j \{\vec{\lambda}_{0,i} \mid i\in\N\}$ is an orthonormal basis for $W_j=D^j W_0$.
Define
\begin{align*}
  \vec{a}_{j,i}         &\equiv D^j \vec{a}_{0,i}, j\in\Z, i\in\J_1.\\
  \vec{b}_{j,i}         &\equiv D^j \vec{b}_{0,i}, j\in\Z, i\in\J_2.\\
  \vec{\lambda}_{j,i}   &\equiv D^j \vec{\lambda}_{0,i}, j\in\Z, i\in\N.
\end{align*}
Define
\begin{align*}
  \Lambda               &= \{\vec{b}_{j,i} \mid j\in\Z, i\in\J_2 \}.\\
  \Phi                  &=\{\vec{\lambda}_{j,i} \mid j\in\Z, i\in\N\}.
\end{align*}
It is clear that
\begin{align*}
    \vec{a}_{j,i} &\in Q W_j, j\in\Z, i\in\J_1.\\
    \vec{b}_{j,i} &\in Q^\bot W_j, j\in\Z, i\in\J_2.\\
    \vec{\lambda}_{j,i} &\in  W_j, j\in\Z, i\in\N.
\end{align*}
Since $D^j$ is a unitary operator, and the spaces $W_j$ are mutually orthogonal and sum to $\H$, the set
$\Phi$ is an orthonormal basis for $\H$. For a point $\vec{x}\in\H$, we have
\begin{align}\label{eq:xbasis}
       \vec{x} = \sum_{j\in\Z,i\in\J_1} \la \vec{x}, \vec{a}_{j,i}  \ra \vec{a}_{j,i} + \sum_{j\in\Z,i\in\J_2} \la \vec{x}, \vec{b}_{j,i}  \ra \vec{b}_{j,i}
       = \sum_{j\in\Z,i\in\N} \la \vec{x}, \vec{\lambda}_{j,i}  \ra \vec{\lambda}_{j,i}.
\end{align}
Let $\vec{x}\in\K$, we have $\la \vec{x}, \vec{a}_{j,i}  \ra = 0$. By the above Equation \eqref{eq:xbasis}, we have
\begin{align*}
    \vec{x} = \sum_{j\in\Z,i\in\J_2} \la \vec{x}, \vec{b}_{j,i}  \ra \vec{b}_{j,i}, \forall \vec{x}\in\K.
\end{align*}
The set $\Lambda = \{\vec{b}_{j,i}\mid j\in\Z, i\in\J_2\}$ is an orthonormal basis for $\K$.

It is also clear that
\begin{lemma}
 Let $\vec{x}\in V_n$. Then
    \begin{align}\label{eq:vn}
        \vec{x} = \sum_{j\leq n, i\in\N} \la \vec{x}, \vec{\lambda}_{j,i} \ra \vec{\lambda}_{j,i}
        =\sum_{j\leq n, i\in\J_1} \la \vec{x}, \vec{a}_{j,i} \ra \vec{a}_{j,i}+
        \sum_{j\leq n, i\in\J_2} \la \vec{x}, \vec{b}_{j,i} \ra \vec{b}_{j,i}
    \end{align}
\end{lemma}

Define
\begin{align*}
    \sigma_{j,i}    &= \iota(\vec{b}_{j,i}), j\in\Z, i\in\J_2.\\
    \Theta          &= \{\sigma_{j,i} \mid j\in\Z, i\in\J_2 \}.
\end{align*}
It is clear that
\begin{align}\label{eq:base}
\Xi(\vec{b}_{j,i})=\Xi_{|\K}(\vec{b}_{j,i})=\iota(\vec{b}_{j,i})=\sigma_{j,i}, j\in\Z, i\in\J_2.
\end{align}
Notice that $\iota$ is an continuous isomorphism from $\K$ to $\H$, the set $\Theta=\iota(\Lambda)$,
 and $\Lambda$ is an orthonormal basis of $\K$. The set $\Theta$ is an unconditional basis
for $\H$. An element $\vec{y}\in\H$ has the form
\begin{align*}
    \vec{y}= \sum_{j\in\Z,i\in \J_2} \beta_{j,i} \sigma_{j,i}.
\end{align*}

\begin{lemma}\label{lm:one}
Let $\vec{y}\in\Xi(V_n)$  Then $\vec{y}$ has the form
\begin{align}
    \vec{y}= \sum_{j\leq n,i\in \J_2} \beta_{j,i} \sigma_{j,i}.
\end{align}
Other words, for the dual basis $\sigma_{j,i} ^*$, we have
\begin{align*}
    \sigma_{j,i} ^* (\vec{y}) = 0, \forall j>n\textrm{ and } i\in\J_2.
\end{align*}
\end{lemma}
\begin{proof}
Assume $\vec{y}=\Xi (\vec{x})$ for some $\vec{x}\in V_n$.
By Equation \eqref{eq:vn},
\begin{align*}
\vec{x} =\sum_{j\leq n, i\in\J_1} \la \vec{x}, \vec{a}_{j,i} \ra \vec{a}_{j,i}+
        \sum_{j\leq n, i\in\J_2} \la \vec{x}, \vec{b}_{j,i} \ra \vec{b}_{j,i}.
\end{align*}
Notice that the elements $\vec{a}_{j,i},j \in\Z, i\in\J_1$ are in the kernel of $\Xi$,
\begin{align*}
    \vec{y} &=\Xi \left(\sum_{j\leq n, i\in\J_1} \la \vec{x}, \vec{a}_{j,i} \ra \vec{a}_{j,i}+
        \sum_{j\leq n, i\in\J_2} \la \vec{x}, \vec{b}_{j,i} \ra \vec{b}_{j,i}\right)\\
            &=\Xi \left(\sum_{j\leq n, i\in\J_2} \la \vec{x}, \vec{b}_{j,i} \ra \vec{b}_{j,i}\right)\\
            &=\sum_{j\leq n, i\in\J_2} \la \vec{x}, \vec{b}_{j,i} \ra \ \Xi \left(\vec{b}_{j,i}\right)\\
            &=\sum_{j\leq n, i\in\J_2} \la \vec{x}, \vec{b}_{j,i} \ra \ \sigma_{j,i},
\end{align*}
by Equation \eqref{eq:base}.
We have
\begin{align*}
    \vec{y} =\sum_{j\leq n, i\in\J_2} \la \vec{x}, \vec{b}_{j,i} \ra \ \sigma_{j,i}.
\end{align*}
Since  $\{\sigma_{j,i}\}$ is a unconditional basis for $\H$, the coefficients $\beta_{j,i}$ for $\vec{y}$ are unique.
$\beta_{j,i}=\la \vec{x}, \vec{b}_{j,i} \ra, j\leq n$. Also $\beta_{j,i}=0$ when $j>n$.

\end{proof}

Next we have
\begin{lemma}
\begin{align}\label{eq:cap0a}
    \bigcap _{n\in\Z} \Xi(V_n)=\{0\}.
\end{align}
\end{lemma}
\begin{proof}
Let $\vec{y} \in \bigcap _{n\in\Z} \Xi(V_n)$.
Then $y\in\Xi(V_n)$ for each $n\in\Z$. Notice that $\{\sigma_{j,i}\}$ is a Schauder basis for $\H$,
\begin{align*}
    \vec{y}=\sum_{j\in\Z,i\in\J_2} \beta_{j,i} \sigma_{j,i},
\end{align*}
the coefficients $\{\beta_{j,i}\}$ are unique for $\vec{y}$. Let $\beta_{j,i}$ be one of the coefficient for some $j\in\Z,i\in\J_2$.
Let $n=j-1$. Since $\vec{y}\in\Xi(V_n)$, and $j>n$, by Lemma \ref{lm:one} $\beta_{j,i} = 0$. This implies that $\vec{y}=0$.

\end{proof}

We have
\begin{lemma}
\begin{align}\label{eq:closure}
    \overline{\Xi (V_0)} &= \Xi (V_0).
\end{align}
\end{lemma}
\begin{proof}
Let $\vec{y}_0\in \H\backslash \Xi (V_0)$. We will show that $\vec{y}_0$ is an exterior point of $\Xi (V_0)$.
It suffices to show that the distance from $\vec{y}_0$ to $\Xi (V_0)$ is positive.
Notice that $\{\sigma_{j,i}\}$ is a Schauder basis for $\H$.
Let $\{\sigma_{j,i}^*\}$ be the associated dual basis.
  We have
\begin{align*}
    \vec{y}_0=\sum_{j\in\Z,i\in\J_2} \gamma_{j,i}\sigma_{j,i},
\end{align*}
for some coefficients $\gamma_{j,i}$.
Since $\vec{y}_0$ in $\H$ but not in $\Xi (V_0)$, $\gamma_{j_0,i_0}\neq 0$ for some $j_0 >0, i_0\in\J_2 $.
Let $\sigma_{j_0,i_0}^*$ be element in the dual basis with index $j_0,i_0$.  
Let $\vec{y}$ be an element in $\Xi (V_0)$. By Lemma \ref{lm:one},
$\sigma_{j_0,i_0}^*(\vec{y})=0$ for each $\vec{y}\in \Xi(V_0)$ but $\sigma_{j_0,i_0}^*(\vec{y}_0)=\gamma_{j_0,i_0}\neq 0$.
\begin{align*}
     |\gamma_{j_0,i_0}|=\left|\sigma_{j_0,i_0}^*\left(\vec{y}_0-\vec{y}\right)\right|\leq \| \vec{\sigma}_{{j_0,i_0}}^* \|\cdot \| \vec{y}_0-\vec{y} \|.
\end{align*}
This implies
\begin{align*}
     \| \vec{y}_0-\vec{y} \| \geq \frac{|\gamma_{j_0,i_0}|}{\|\sigma_{{j_0,i_0}}^*\|}>0, \forall \vec{y}\in \Xi(V_0).
\end{align*}
So the open ball $B(\vec{y}_0, r)$ with $r=\frac{|\gamma_{j_0,i_0}|}{2\|\sigma_{{j_0,i_0}}^*\|}$ must be disjoint with
$\vec{y}\in \Xi(V_0)$.
So $\Xi(V_0)$ is closed.

\end{proof}

Now we will prove our conclusion in this paper.
Since $\Xi$ is linear, and by Equation \eqref{eq:pcommutant}, we have
\begin{align*}
    &\textrm{span} \{D^{\textrm{-}m}T^\ell \psi \mid m\in\Z, m\geq 0, \ell\in\Z\}\\
    &=  \textrm{span} \{D^{\textrm{-}m}T^\ell \Xi\eta \mid m\in\Z, m\geq 0, \ell\in\Z\}\\
    &=  \Xi \left(\textrm{span} \{D^{\textrm{-}m}T^\ell \eta \mid m\in\Z, m\geq 0, \ell\in\Z\}\right).
\end{align*}
This implies
\begin{align*}
    V(\psi) &=\overline{\textrm{span} \{D^{\textrm{-}m}T^\ell \psi \mid m\in\Z, m\geq 0, \ell\in\Z\}}\\
    &=   \overline{\Xi \left(\textrm{span} \{D^{\textrm{-}m}T^\ell \eta \mid m\in\Z, m\geq 0, \ell\in\Z\}\right)}\\
    &\subseteq   \overline{\Xi \left(\overline{\textrm{span} \{D^{\textrm{-}m}T^\ell \eta \mid m\in\Z, m\geq 0, \ell\in\Z\}}\right)}\\
    &= \overline{\Xi (V_0)}.
\end{align*}
Thus
\begin{align}\label{eq:a}
    V(\psi) &\subseteq \overline{\Xi (V_0)}.
\end{align}
By Equation \eqref{eq:closure} we have
\begin{align}\label{eq:a}
    V(\psi) &\subseteq \overline{\Xi (V_0)}=\Xi(V_0).
\end{align}
So  for each $n\in\Z$, $D^n V(\psi) \subseteq D^n \Xi (V_0)= \Xi D^n V_0=\Xi(V_n)$.
\begin{align*}
        D^n V(\psi) \subseteq \Xi(V_n), \forall n\in\Z.
\end{align*}
Therefore, by Equation \eqref{eq:cap0a}
\begin{align*}
    \bigcap_{n\in\Z} D^n V(\psi) \subseteq \bigcap_{n\in\Z} \Xi(V_n) =\{0\}.
\end{align*}
So,  Equation \eqref{eq:baggett} has been established.

\end{document}